\documentclass[11pt, reqno]{amsart}
\title[Set-theoretic Yang-Baxter $\&$ reflection equations]{Algebraic structures in set-theoretic Yang-Baxter 
$\&$  reflection equations}
\author[Anastasia Doikou]{Anastasia Doikou}

\address[A. Doikou] {Department of Mathematics, Heriot-Watt University,
Edinburgh EH14 4AS, and The Maxwell Institute for Mathematical Sciences, Edinburgh}
\email{A.doikou@hw.ac.uk}

\usepackage[english]{babel}
\usepackage[utf8]{inputenc}
\usepackage{amsmath, amssymb, bm}
\usepackage{graphicx}
\usepackage{subcaption}
\usepackage[colorinlistoftodos]{todonotes}
\usepackage{indentfirst}

\newcommand{\hiddenpower}[2] { \ifnum \numexpr#2=1 #1 \else #1^#2 \fi }
\numberwithin{equation}{section}

\newcommand{\cal}{\mathcal}

\usepackage{xparse}
\ExplSyntaxOn
\newcounter{diff_order}
\newcounter{diff_power}

\newcommand{\rawdiff}[3]
{
	\setcounter{diff_order}{0}
	\clist_map_inline:nn{#3}{\stepcounter{diff_order}}
	
	\frac{\hiddenpower{#1}{\thediff_order} #2}
	{
		\def\old_var{DefaultValue}
		\setcounter{diff_power}{0}
		
		\clist_map_inline:nn{#3}
		{
			\def\new_var{##1}
			\ifnum \thediff_power=0
				\stepcounter{diff_power}
			\else
				\tl_if_eq:NNTF \new_var \old_var
				{\stepcounter{diff_power}}
				{
					#1 \hiddenpower{\old_var}{\thediff_power}
					\setcounter{diff_power}{1}
				}
			\fi

			\def\old_var{##1}
		}
		
		#1 \hiddenpower{\old_var}{\thediff_power}
	}
}
\ExplSyntaxOff

\newlength{\bibitemsep}
\newlength{\bibparskip}
\let\oldthebibliography\thebibliography
\renewcommand\thebibliography[1]{%
  \oldthebibliography{#1}%
  \setlength{\parskip}{\bibitemsep}%
  \setlength{\itemsep}{\bibparskip}%
}

\newtheorem{thm}{Theorem}[section]
\newtheorem{lemma}[thm]{Lemma}
\newtheorem{cor}[thm]{Corollary}
\newtheorem{pro}[thm]{Proposition}
\newtheorem{defn}[thm]{Definition}

\newtheorem{example}[thm]{Example}

\newtheorem{rem}[thm]{Remark}

\newcommand{\sz}[2]{\sigma_{#1}(#2)}

\newcommand{\tz}[2]{\tau_{#1}(#2)}

\begin{document}






\begin{abstract}
\noindent  
We present resent results regarding invertible, non-degenerate solutions of the set-theoretic Yang-Baxter and reflection equations. 
We  recall the notion of braces and we present  and prove various fundamental properties required for the solutions
of the set theoretic Yang-Baxter equation. We then restrict our attention on involutive solutions and consider $\lambda$ parametric 
set-theoretic solutions of the Yang-Baxter equation and we extract 
the associated quantum algebra. We also discuss the notion of the Drinfeld twist for involutive solutions and their relation to the Yangian.
We next focus on reflections and we derive the associated defining algebra relations for $R$-matrices being  
Baxterized solutions of the symmetric group.  
We show that there exists a ``reflection'' finite sub-algebra for some special choice of reflection maps.  \\
\end{abstract}

\maketitle

\date{}
\vskip 0.4in



\section*{Introduction}

\noindent Yang-Baxter equation (YBE)  is a central object in the framework 
of quantum integrable systems \cite{Baxter, Yang} and quantum algebras \cite{Drinfeld, Jimbo}.   The notion of  
set-theoretic solutions to the Yang-Baxter equation, was first suggested by Drinfeld \cite{Drin} and since then 
a great deal of research has been devoted to this issue (see for instance \cite{ESS, chin}), yielding also significant connections 
between the set-theoretic Yang-Baxter equation and geometric crystals \cite{Eti, BerKaz}, 
or soliton cellular automatons \cite{TakSat, HatKunTak}.
From a purely algebraic point of view the theory of braces  was established by  
W.  Rump  to describe all finite involutive,  set-theoretic solutions of  the Yang-Baxter equation \cite{[25], [26]}, 
whereas skew-braces were then developed in \cite{GV}  to describe non-involutive solutions.

Yang-Baxter equation is the key equation for the construction of integrable systems with (quasi)-periodic boundary conditions.
However, in order to be able to incorporate boundary conditions to these
systems that preserve integrability the boundary Yang-Baxter or reflection equations is also needed \cite{Cherednik, Sklyanin}. 
The set-theoretic reflection equation together with the first examples of solutions first 
appeared in \cite{CauZha}, while a more systematic study and a classification inspired by maps appearing 
in integrable discrete systems presented 
in \cite{CauCra}. Other solutions were also considered and used 
within the context of cellular automata \cite{KuOkYa}, whereas in \cite{SmoVenWes, Katsa} 
the theory of braces was used  
to produce families of new solutions to the reflection 
equation, and in \cite{DeCommer} skew braces were used to
produce reflections. Moreover, key connections between  set-theoretic solutions, quantum 
integrable systems and the associated 
quantum algebras were uncovered in \cite{DoiSmo, DoiSmo2, Doikoutw} and \cite{DoGhVl, DoiRyb}. Here, we present some 
of the fundamental recent results, which have opened new paths on the study of quantum integrable systems coming 
from set-theoretic solutions of the Yang-Baxter and reflection equations.

One of the main aims of this article is to present basic results regarding (skew) braces and the set-theoretic Yang-Baxter and reflection equations in  a detailed,  pedagogical way so that non-experts can follow through. Specifically, the main objectives of this article are:
\begin{itemize}
\item To introduce braces as the main algebraic structure underlying set-theoretic solutions of the Yang-Baxter equation.
\item To study the quantum algebra associated to set-theoretic solutions of the Yang-Baxter equation. 
\item To Baxterize involutive set-theoretic solutions and obtain $\lambda$-dependent solutions, which give rise to a certain affine algebra. 
The twisting of  involutive solutions and their quantum algebra and their connection to the Yangian is also discussed.
\item To study fundamental solutions of the reflection equation for set-theoretic $R$-matrices and construct the reflection algebra using Baxterized set-theoretic $R$ and reflection matrices.
\end{itemize}

The outline of the study is as follows. 

\begin{enumerate}
 \item In Section 1 we recall the notion of the set-theoretic YBE and we 
introduce the algebraic structure underlying the solutions of the YBE called braces \cite{[25], [26]}.  Then starting from invertible 
solutions of the YBE we reconstruct a slightly more general structure called near brace and vise versa using near braces we define
suitable bijective maps tha satisfy the YBE (see also \cite{DoiRyb}). In Subsection 1.1 we use Baxterized solutions coming 
from involutive set-theoretic solutions of the YBE and study the associated quantum algebra, which surprisingly turns out to be a twisted 
version of the Yangian \cite{DoiSmo}-\cite{DoiRyb}.   A brief discussion on the admissible set-theoretic twist is also presented. 

\item In Section 2 we discuss the set-theoretic reflection equation.
More precisely,  we review
some recent results on solutions of the set-theoretic reflection equation \cite{CauCra, SmoVenWes}. 
In Subsection 2.1 we derive the associated defining algebra relations (reflection algebra) for Baxterized solutions of the YB and reflection equations
and we show that there exist a finite sub-algebra of the reflection algebra for some special choice of ``reflection elements''  \cite{DoiSmo2}.
\end{enumerate}

\section{The set-theoretic Yang-Baxter equation}

\noindent Let $X=\{{\mathrm x}_{1}, \ldots, {\mathrm x}_{{\cal N}}\}$ be a set and ${\check r}:X\times X\rightarrow X\times X$.
 Denote \[{\check r}(x,y)= \big (\sigma _{x}(y), \tau _{y}(x)\big ).\] 
We say that $\check r$ is non-degenerate if $\sigma _{x}$ 
and $\tau _{y}$ are bijective functions. Also, the solution $(X, \check r)$  is involutive if $\check r ( \sigma _{x}(y), \tau _{y}(x)) = 
(x, y)$, ($\check r \check r (x, y) = (x, y)$). We focus on non-degenerate, invertible solutions of the set-theoretic braid equation:
\[({\check r}\times id_{X})(id_{X}\times {\check r})({\check r}\times id_{X})=(id_{X}\times {\check r})({\check r}
\times id_{X})(id_{X}\times {\check r}).\]

$ $
Let us now recall the role of skew braces in the derivation of non-degenerate, set-theoretic solutions of the Yang-Baxter equation.
Let us first give the definitions of skew braces. 

\begin{defn}\cite{[25], [26], [6]}.
A {\it left skew brace} is a set $B$ together with two group operations $+,\circ :B\times B\to B$, 
the first is called addition and the second is called multiplication, such that for all $ a,b,c\in B$,
\begin{equation}\label{def:dis}
a\circ (b+c)=a\circ b-a+a\circ c.
\end{equation}
If $+$ is an abelian group operation $B$ is called a {\it left brace}.
Moreover, if $B$ is a left skew brace and for all $ a,b,c\in B$ $(b+c)\circ a=b\circ a-a+c \circ a$, then $B$ is called a 
{\it skew brace}. Analogously if $+$ is abelian and $B$ is a skew brace,  then $B$ is called a {\it brace}.
\end{defn}
The additive identity of a brace $A$ will be denoted by $0$ and the multiplicative identity by $1$.  
In every skew brace $0=1$.

We state below the fundamental Theorems on non-degenerate solutions of the set-theoretic solutions of the YBE.
Rump showed the following  theorem for  involutive, non degenerate, set-theoretic solutions.
\begin{thm}\label{Rump}(Rump's theorem, \cite{[25], [26]}).  
Assume  $(B, +, \circ)$ is a left brace. If the map  $\check r_B: B\times B \to B \times B$ is defined as 
${\check r}_B(x,y)=(\sigma _{x}(y), \tau _{y}(x))$, where $\sigma _{x}(y)=x\circ y-x$, 
$\tau _{y}(x)=t\circ x-t$, and $t $ 
is the inverse of $\sigma _{x}(y)$ in the circle group $(B, \circ ),$ then  $(B, \check r_B)$ is an involutive,  
non-degenerate solution of the braid equation.\\
Conversely,  if $(X,\check r)$ is an involutive, non-degenerate solution of the braid equation, then there exists a left brace $(B,+, \circ)$ 
(called an  underlying brace of the solution $(X, \check r)$) such that $B$ contains $X,$ $\check r_B(X\times X )\subseteq X \times X,$
and the map $\check r$ is equal to the restriction of $\check r_B$ to $X \times X.$ Both the additive $(B, +)$ 
and multiplicative $(B,\circ)$ groups of the left brace $(B,+, \circ)$ are generated by $X.$
\end{thm}
\begin{rem} [Rump]\label{nilpotent}
 Let $(N, +, \cdot)$ be an associative ring.  If for $a,b\in N$ we define 
 \[a\circ b=a\cdot b+a+b,\] then $(N, +, \circ )$ is a brace if and only if $(N, +, \cdot)$ is a radical ring.
\end{rem}
Guarnieri and Vendramin \cite{GV}, generalized Rump's result to left skew braces and non-degenerate,  
non-involutive solutions.
\begin{thm}[{\it Theorem \cite{GV}}]\label{thm:GV}
Let $B$ be a left skew brace, then the map $\check{r}_{GV}:B\times B\to B\times B$ given for all $ a,b\in B$ by
$$
\check{r}_{GV}(a,b)=(-a+a\circ b,\ (-a+a\circ b)^{-1}\circ a\circ b)
$$
is a non-degenerate solution of set-theoretic YBE.  
\end{thm}
We will show below a slight generalization of the theorems above by introducing the notion of a (commutative) {\it near brace}. This generalization concerns the
condition $0=1,$ which is always true for (skew) braces, whereas in the case of near braces 
this condition does not necessarily hold any more.  In fact,  skew braces can be seen as a special case of near braces.

First we will show  that given any non-degenerate,  invertible solution 
${\check r}(x,y)=(\sigma _{x}(y), \tau _{y}(x))$  a  near brace structure can be reconstructed and vise versa given a near brace
every ${\check r}(x,y)=(\sigma _{x}(y), \tau _{y}(x))$, 
where $\sigma _{x}(y)=x\circ y-x$,  $\tau _{y}(x)=\sigma _{x}(y)^{-1} \circ x \circ y$
is a solution of the YBE.  

We prove below these fundamental statements.
We first review the constraints arising by requiring $(X, \check r)$ ($\check r (x,y) = (\sigma_{x}(y)),\tau_{y}(x))$ to be a solution 
of the braid equation \cite{Drin,  ESS, [25], [26]}. Let,
\[({\check r}\times \mbox{id})(\mbox{id}\times {\check r})({\check r}\times \mbox{id})(\eta, x, y)=(L_1, L_2, L_3),\]
\[(\mbox{id}\times {\check r})({\check r}\times \mbox{id})(\mbox{id}\times {\check r})(\eta,x,y)= (R_1, R_2, R_3),\]
where, after using the forms of the set theoretic solution we obtain:
\begin{eqnarray}
L_1 = \sigma_{\sigma_{\eta}(x)}(\sigma_{\tau_{x}(\eta)}(y)),\quad L_2 = 
\tau_{\sigma_{\tau_x(\eta)}(y)}(\sigma_{\eta}(x)), \quad L_3 =\tau_y(\tau_x(\eta)),\nonumber
\end{eqnarray}
\begin{eqnarray}
R_1=\sigma_{\eta}(\sigma_x(y)), \quad R_2=\sigma_{\tau_{\sigma_x(y)}(\eta)}(\tau_{y}(x)), \quad 
R_3= \tau_{\tau_{y}(x)}(\tau_{\sigma_{x}(y)}(\eta)). \nonumber
\end{eqnarray}
And by requiring $L_i =R_i,$ $i\in \{1,2,3\}$ we obtain the following fundamental constraints for the associated maps:
\begin{eqnarray}
&&  \sigma_{\eta}(\sigma_x(y))= \sigma_{\sigma_{\eta}(x)}(\sigma_{\tau_{x}(\eta)}(y)), \label{C1}\\
&& \tau_y(\tau_x(\eta)) = \tau_{\tau_{y}(x)}(\tau_{\sigma_{x}(y)}(\eta)), \label{C2}\\
&&  \tau_{\sigma_{\tau_x(\eta)}(y)}(\sigma_{\eta}(x))= \sigma_{\tau_{\sigma_x(y)}(\eta)}(\tau_{y}(x)). \label{C3}
\end{eqnarray}

We start with the first part of our construction.  As mentioned above we are going to slightly 
generalize the structure of the skew brace 
by introducing the near brace \cite{DoiRyb}, which can be reconstructed from any non-degenerate solution 
of the set-theoretic braid equation.
For the rest of the subsection we consider $X$ to be a set and there exists a binary group operation 
$\circ :X\times X\to X$,  with a neutral element $1\in X$ and an inverse $x^{-1} \in X,$ 
$\forall x \in X.$  There also exists a bijective function $\sigma_x: X \to X,$  $\forall x \in X,$
such that $y \mapsto \sigma_x(y).$ 
We then define another binary operation $+: X\times X \to  X,$ such that
\begin{equation}
y + x:= x \circ \sigma_{x^{-1}}(y) \label{mapzz}
\end{equation}
and we assume that it is {\it associative} (this assumption leads to certain constraints, 
for more details the interested reader is referred to \cite{DoiRyb}).  
Notice that in general $+$ is a non commutative operation,  but in the case of involutive solutions,  
it turns out to be commutative as will be clear in 
Theorem \ref{theorem1} (see also \cite{DoiRyb}).

We focus on non-degenerate, invertible solutions $\check r.$  
Given that $\sigma_x$ and $\tau_y$ are bijections the inverse maps also exist such that
\begin{equation}
\sigma^{-1}_x(\sigma_x(y)) = \sigma_x(\sigma^{-1}_x(y))= y ,  \quad  \tau^{-1}_{y}(\tau_y(x)) = \tau_{y}(\tau^{-1}_y(x)) =x \label{inv1}
\end{equation}
Let
the inverse $\check r^{-1}(x,y) = (\hat \sigma_x(y), \hat \tau_y(x))$ exist with $\hat \sigma_x,\ \hat \tau_y$ 
being also bijections,  that satisfy:
\begin{equation}
\sigma_{\hat \sigma_x(y)} (\hat \tau_y(x)) = x=\hat \sigma_{\sigma_x(y)} (\tau_y(x)), \quad 
\tau_{\hat \tau_y(x)}(\hat \sigma_x(y)) = y = \hat \tau_{\tau_y(x)}(\sigma_x(y)).   \label{ide1}
\end{equation}
Taking also into consideration (\ref{inv1}) and (\ref{ide1})  and that  $\sigma_x,\tau_y$ 
and $\hat \sigma_x,  \hat\tau_y$ are bijections, we deduce:
\begin{equation}
\hat \sigma^{-1}_{\sigma_x(y)}(x) = \tau_y(x),  \quad \hat \tau^{-1}_{\tau_y(x)}(y) = \sigma_x(y).
\end{equation}

We assume that the map $\hat \sigma$ appearing in the inverse matrix $\check r^{-1}$ has the general form
\begin{equation}
\hat \sigma_{x}(y)  =  x \circ (x^{-1} + y ).  \label{mapzz2}
\end{equation}
The origin of the above map comes from the definition: $x+y:=  x \circ \hat \sigma_{x^{-1}}(y).$ 
The derivation of $\check r$ goes hand in hand with the derivations of $\check r^{-1}$
(see details in \cite{DoiRyb} and later in the text when deriving a generic $\check r$ and its inverse).  
In the involutive case the two maps coincide and $x+y = y+x.$ However, 
for any non-degenerate,  non-involutive solution 
both bijective maps $\sigma_x, \hat \sigma_x$ should be 
considered together with the fundamental conditions (\ref{ide1}).

We present below a series of useful lemmas that will lead to the main theorem. 
We  consider for the sake of simplicity only finite sets here.
\begin{rem} \label{rr1}  
Let us first remind a known fact. We recall that $\sigma_x$ is a bijective function, then using (\ref{mapzz})
$\sigma_x(y_1) = \sigma_x(y_2) \Leftrightarrow   y_1  + x^{-1} = y_2  + x^{-1},$\\
which automatically suggest right cancellation in $+.$ Similarly,  $\hat \sigma_x$ is a bijective function,  which leads  
also to left cancellation in $+.$
\end{rem}

\begin{lemma} \label{bi}
For all $y\in X$, the operation $+x:X\to X$ is a bijection. 
\end{lemma}
\begin{proof}
Let $y_1,y_2\in X$ be such that $y_1+x=y_2+x$, then
$$
x\circ \sigma_{x^{-1}}(y_1)=x\circ \sigma_{x^{-1}}(y_2)\implies  \sigma_{x^{-1}}(y_1)=\sigma_{x^{-1}}(y_2),
$$
since $\circ$ is a group operation and $\sigma_{x^{-1}}$ is injective, we get that $y_1=y_2$ and $+x$ is injective for any $x\in X$. 
For finite sets injectivity is sufficient to guarantee bijectivity. 
Thus $+x$ is a bijection.  Similarly,  from the bijectivity of $\hat \sigma_x$ and (\ref{mapzz2}) we show that $x+$ is also a bijection.
\end{proof}

We now introduce the notion of a neutral elements in $(X,+)$
\begin{lemma} 
Let $0_x \in X$ such that $x+ 0_x = x,$ $\forall x \in X,$ then $0_x = 0_y =0,$ $\forall x, y \in X$ 
and $0$ is a unique right  neutral element.  The right neutral element $0$ is also left neutral element. 
\end{lemma}
\begin{proof}
Let $0_x\in X$ exists  $\forall x \in X, $ and recall the definition of $+$ in (\ref{mapzz}), then
\begin{equation}
x+0_x  = x \Rightarrow y+x +0_x = y+x, \nonumber\\
\end{equation}
but also $~y+x + 0_{y+x} = y+x.$

The last two equations lead to $y+x + 0_x  = y+x +0_{y+x},$ and due to Remark \ref{rr1} 
left cancellation holds,  thus after 
setting $y+x = w$ and recalling Remark \ref{bi}, we deduce
$0_x = 0_w,$ $ \forall x,w \in X.$

Moreover,  $y+0 =y \Rightarrow y+0 +x = y+x$ and due to associativity and 
right cancellation (Remark \ref{rr1}) for $+$ we deduce $0+x = x.$
\end{proof}

\begin{lemma} Let $0$ be the neutral element in $(X,+)$,  then
$\forall x \in X,$ $\exists -x \in X,$ such that $-x  + x = 0$ (left inverse).  
Moreover, $-x \in X$ is a right inverse,  i.e.  $x -x = 0$ $\forall x \in X.$  That is $(X,+,0)$ is a group.
\end{lemma}
\begin{proof} 
Observe that due to bijectivity of $\sigma_x$, we can consider $-x:=\sigma_{x^{-1}}^{-1}(x^{-1}\circ 0)$. 
Simple computation shows it is a left inverse,
$$
-x+x=x\circ \sigma_{x^{-1}} (\sigma_{x^{-1}}^{-1}(x^{-1}\circ 0))=0.
$$
By associativity we deduce that $x+(-x)+x=0+x$, we get that $x+(-x)=0$, and $-x$ is the inverse.  
\end{proof}
By having assumed associativity of $+$ we have shown that $(X, +)$ is also a group.

\begin{rem} \label{remr2} It is worth noting that the usual distributivity 
rule does not apply.  Indeed let $(X,+)$ and $(X, \circ)$ be both groups.  We now consider the usual distributivity rule: 
 $ a\circ ( 0 +a^{-1}) =1 \Rightarrow a\circ 0 + 1 = 1 \Rightarrow a\circ 0 = 0,$ 
but given that $0 \in (X, \circ)$ is invertible,  the latter leads to $a = 1,$ $\forall a \in X,$ which is not true.
We should therefore introduce a more general distributivity rule in this context.  
Indeed, henceforth we assume $a\circ (b +c) = a \circ b +\phi(a) + b\circ c,$ $\phi(a)$ to be identified.
\end{rem}

\begin{thm} \label{theorem1}
Let $(X, \circ)$ be a group and $\check r: X \times X \to X \times X$ be 
such that $\check r (x, y) = (\sigma_x(y), \tau_x(y))$ is a non-degenerate, invertible solution of the 
set-theoretic braid equation and $(X,+)$($+$ is  defined in (\ref{mapzz})) is a group.  Moreover,  we assume that:
\begin{itemize}
\item\label{prop:con1} There exists $\phi:X\to X$ such that for all $a,b,c\in X$ $a\circ (b+c)=a\circ b+\phi(a)+a\circ c.$
\item The neutral element $0$ of $(X, +)$ has a left and right distributivity,  i.e.  $(a+b) \circ 0  = a \circ 0 +\hat \phi(0) + b \circ 0.$ 
\end{itemize}
Then for all $a,b,c\in X$ the following statements hold:
\begin{enumerate}
\item $\phi(a) = - a\circ 0$ and $~\widehat{\phi}(h) = -0\circ 0$, \label{prop:main2}
\item  $\sigma_a(b)=a\circ b-a\circ 0+1.$\label{prop:main3}
\item $a-a\circ 0 =1=-a\circ a +a$ and (i) 1+a  =a+1 (ii) $0\circ 0=-1$  (iii) $1+1 = 0^{-1}.$  \label{prop:main4} 
\item $\hat \sigma_a(b) \circ \hat \tau_b(a)= a\circ b  = \sigma_a(b) \circ\tau_b(a).$ \label{prop:main4b} 
\end{enumerate}
\end{thm}
\begin{proof}
$ $
\begin{enumerate}

\item This is straightforward: $a = a \circ (1+0) \Rightarrow a = a + \phi(a) + a \circ 0 \Rightarrow \phi(a) = - a\circ 0.$ 
Similarly, $ (1+0) \circ 0 = 0 \Rightarrow \hat \phi(0) = - 0 \circ 0.$ The distributivity can be checked by recalling the definition of + (\ref{mapzz}). 
\item We recall the definition of $+$ in (\ref{mapzz}) and consider the distributivity rule 
$a\circ(b +c ) = a\circ b -a\circ 0 + b\circ c.$
We then obtain
\begin{equation}
\sigma_a(b) = a \circ (b + a^{-1} ) \Rightarrow \sigma_a(b) = a\circ b - a\circ 0 +1. \label{dist2}
\end{equation}
 The validity of the distributivity rule can be checked by comparing the LHS and RHS in: 
$a\circ (c +b) = a \circ b \circ \sigma_{b^{-1}}(c).$

\item Due to the fact that $\check r$ satisfies the braid equation we may  employ (\ref{C1}) 
and the distributivity rule (see also (\ref{dist2})):
\begin{eqnarray}
\sigma_a(\sigma_b(c)) &= & a \circ \sigma_b(c) - a\circ 0 +1  \nonumber\\
&=& a\circ b \circ(c +b^{-1}) - a\circ 0 +1\nonumber\\
& =& a\circ b \circ c - a\circ b\circ 0 + a - a\circ 0 +1. \label{basic1} \nonumber
\end{eqnarray}
But due to condition (\ref{C1}) and by setting $c = 0,$ we deduce that
$a- a\circ 0 = \zeta,$ $~\forall a \in X$   ($\zeta$ is a fixed element in $X$), but for $a=1$ 
we immediately obtain $\zeta = 1,$
i.e. 
\begin{equation}
a -a \circ 0 =1. \label{conditiona}
\end{equation} 

Similarly,  $\check r^{-1}$ satisfies the braid equation, then via  (\ref{C1}) for $\hat \sigma$
and the distributivity rule for (see also (\ref{dist2})) we conclude that $-a\circ 0 +a =1.$ 

(i) Via $a-a\circ 0=-a\circ 0 +a =1$ we conclude that $\ a+1 =1+a.$

(ii) By setting $a=0$ in (\ref{conditiona}) we have $0\circ 0 =-1.$

(iii) $0\circ (1+1) = 0\circ 1 - 0\circ 0 + 0\circ 1 \Rightarrow 1+1 = 0^{-1}.$

\item  
\noindent Recall the form of $\hat \sigma_a(b) $ (\ref{mapzz2}), and use the distributivity rules, then 
\begin{equation}
\hat \sigma_a(b) = 1 - a \circ 0  + a\circ b.
\end{equation}

We recall relations (\ref{ide1}) for the maps and also recall that $a-a\circ 0 =1$, then
\begin{eqnarray}
& &\sigma_{\hat \sigma_a(b)}(\hat \tau_{b}(a)) = a 
\Rightarrow \hat \sigma_a(b) \circ \hat \tau_b(a) - \hat \sigma_a(b) \circ 0 + 1 = a \Rightarrow \nonumber\\
& & \hat \sigma_a(b) \circ \hat \tau_b(a) - (1-a\circ  0+ a\circ b) \circ 0+1 =a \nonumber\\
& & \hat \sigma_a(b) \circ \hat \tau_b(a)  - a\circ b  +a\circ 0 -1 +1 +1 =a \Rightarrow\nonumber\\
 & &  \hat \sigma_a(b) \circ \hat \tau_b(a)  -a\circ b+ a =a\  \Rightarrow   \hat \sigma_a(b) \circ \hat \tau_b(a)  =a\circ b. \nonumber
\end{eqnarray}
Similarly,  $\hat \sigma_{ \sigma_a(b)}( \tau_{b}(a)) = a \Rightarrow \sigma_a(b) \circ  \tau_b(a)  = a\circ b.$

Notice that in the special case of involutive solutions $\sigma_x = \hat \sigma_x$ and consequently $(X,+)$ is abelian.
\hfill \qedhere
\end{enumerate}
\end{proof}

We have been able to reconstruct
the algebraic structure underlying invertible,  non-degenerate solutions of the set-theoretic YBE. 
Given the above algebraic construction we may  provide the following  definition.
\begin{defn} \label{def1a}
A {\it near brace} is a set $B$ together with two group operations $+,\circ :B\times B\to B$, 
the first is called addition and the second is called multiplication, such that $\forall a,b,c\in B$,
\begin{equation}
a\circ (b+c)=a\circ b-a\circ 0+a\circ c.\label{cond2}
\end{equation}
We recall that $0$ is the neutral element of the 
$(B, +)$ group and $1$ is the neutral element of the $(B, \circ)$ group.  When $(B,+)$ is abelian then $(B, \circ, +)$ 
is called an abelian near brace.
\end{defn}

When in addition to (\ref{cond2}), condition $a- a\circ 0 =-a\circ 0 +a= 1,$ $\forall a \in B$ 
also holds, then we call the near brace a {\it singular near brace}.

\begin{rem} In the special case where we choose $0=1$ skew braces are recovered (in the abelian case braces are recovered).
In fact, the construction above slightly generalizes previous results on braces and skew braces in the sense that $0=1$ is not required anymore.
\end{rem}

We continue now with the second part of our construction summarized in Theorem \ref{prop:main}. 
We state first a useful Proposition:
\begin{pro}\label{lem:long}
Let $B$ be a near brace and let us denote by $\sigma_a(b):=a\circ b-a\circ 0 +1$ and $\tau_b(a):=
\sigma_a(b)^{-1}\circ a\circ b$, where $a,b\in B$, $\sz{a}{b}^{-1}$ is the inverse of 
$\sz{a}{b}$ in $(B,\circ).$ Then $\forall a,b,c,d\in B$ the following properties hold:
\begin{enumerate}
\item $\sz{a}{b}\circ\tz{b}{a}= a\circ b$\label{lem:long:eq:5}
\item $\sz{a}{\sz{b}{c}}=\sz{a\circ b}{c} +1 $, \label{lem:long:eq:2}
\item $\sigma_{a}(b) \circ \sigma_{\tau_{b}(a)}(c)=  \sigma_a(b\circ c) +1.$ \label{lem:long:eq:6}
\end{enumerate}
\end{pro}

\begin{proof}
Let $a,b,c,d \in B$, then

\begin{enumerate}
\item 
$\sz{a}{b}\circ\tz{b}{a}=\sz{a}{b}\circ \sz{a}{b}^{-1}\circ a\circ b= a\circ b.$

$ $

\item $\sz{a}{\sz{b}{c}}=\sz{a}{b\circ c-b\circ 0 +1}=a\circ (b\circ c-b\circ 0 +1)- a\circ 0 +1$
\begin{eqnarray}
&=& a\circ b\circ c-a\circ b\circ 0+a-a\circ 0 +1 \nonumber \\ 
&=& a\circ b\circ c-a\circ b\circ 0 +1 +1=\sz{a\circ b}{c} +1.\nonumber 
\end{eqnarray}

\item To show (3)  we observe:
$$
\begin{aligned}
\sigma_{a}(b) \circ \sigma_{\tau_{b}(a)}(c)&=\sz{a}{b}\circ(\tz{b}{a}\circ c-
 \tz{b}{a}\circ 0 +1)\\ &=\sz{a}{b}\circ \tz{a}{b}\circ c- \sz{a}{b}\circ \tz{a}{b}\circ 0+\sz{a}{b}\\ &=
a\circ b\circ c-a\circ b\circ 0 + \sz{a}{b}\\ &=a\circ b\circ c-a\circ b\circ 0+ 
(a\circ b-a\circ 0 +1) \\ &=a \circ b \circ c- a\circ 0 +1 +1  = \sigma_a(b\circ c) +1.
\hfill \qedhere
\end{aligned}
$$
\end{enumerate}
\end{proof}

We may now prove the following main theorem (slight generalization of the findings in \cite{[25],  [26]}).
\begin{thm}\label{prop:main}
Let $B$ be a near brace.
Then we can define a map $\check{r}:B\times B\to B\times B$ given by
$$
\check{r}(a,b)=(\sz{a}{b},\tz{b}{a}):=(a\circ b-a\circ 0+1,\ (a\circ b-a\circ 0 +1)^{-1}\circ a\circ b).
$$
 The pair $(B,\check{r})$ 
is a solution of the braid equation.
\end{thm}
\begin{proof} To prove this we need to show that the maps $\sigma, \tau$ satisfy the constraints (\ref{C1})-(\ref{C3}).  
To achieve this we use the properties from Proposition \ref{lem:long}. 

Indeed, from Proposition \ref{lem:long}, (1) and (2), it follows that (\ref{C1}) is satisfied, i.e. 
\[\sz{\eta}{\sz{x}{y}}=\sz{\sz{\eta}{x}}{\sz{\tz{x}{\eta}}{y}}.\]

We observe that
\begin{equation}
\tau_{b}(\tau_a(\eta)) = T\circ \tau_{a}(\eta)\circ b = T \circ t \circ \eta \circ a \circ b = 
T \circ t \circ \eta \circ \sz{a}{b} \circ \tz{b}{a}, \nonumber
\end{equation}
where $T= \sz{\tz{a}{\eta}}{b}^{-1}$ and $t = \sigma_{\eta}(a)^{-1} $ (the inverse in the  circle group). 
Due to the properties of Proposition \ref{lem:long} we then conclude that 
\[\tau_{b}(\tau_a(\eta)) =\tau_{\tz{b}{a}}(\tau_{\sz{a}{b}}(\eta)),\] 
so (\ref{C2}) is also satisfied. 

To prove (\ref{C3}),  we employ (3) of Proposition \ref{lem:long} and then use the definition of $\tau$, 
$$
\sigma_{\tau_{\sigma_x(y)}(\eta)}(\tau_{y}(x))=\sz{\eta\circ x}{y}^{-1} \circ\sigma_{\eta}(x) \circ 
\sigma_{\tau_{x}(\eta)}(y)=\tz{\sz{\tz{x}{\eta}}{y}}{\sz{\eta}{x}}.
$$
Thus, (\ref{C3}) is satisfied, and $\check{r}(a,b)=(\sz{a}{b},\tz{b}{a})$ is a solution of braid equation.

In the special case where the near brace is commutative in $+,$ then the solution is involutive.
\end{proof}

\subsection{Set-theoretic Yang Baxter equation $\&$ quantum groups}

\noindent In this subsection we briefly present some fundamental results on the various links between braces, 
and quantum algebras (see also \cite{DoiSmo}). Recall that we focus on involutive, non-degenerate set-theoretic solutions of the braid equation.

Let $V$ be a vector space of dimension equal to the cardinality of $X$, and with a slight abuse of notation, 
let  $\check r$ also  denote the $R$-matrix associated to the linearisation of ${\check r}$ on 
$V={\mathbb C }X$ (see \cite{LAA} for more details), i.e.
$\check r$  is the ${\cal N}^2 \times{\cal N}^2$ matrix: 
\begin{equation}
\check r = \sum_{x,y,z,w \in X} \check r(x,z|y,w) e_{x,z} \otimes e_{y, w}, 
\end{equation}
where
$e_{x, y}$ is the ${\cal N} \times {\cal N}$  matrix: $(e_{x,y})_{z,w} =\delta_{x,z}\delta_{y,w} $.
Then for the $\check r$-matrix related  to $(X,{\check r})$:  ${\check r}(x,z|y,w)=\delta_{z, \sigma_x(y)} \delta_{w, \tau_y(x)}$.
 Notice that the matrix $\check r:V\otimes V\rightarrow V\otimes V$ satisfies the (constant) Braid equation:
\[({\check r}\otimes I_{V})(I_{V}\otimes {\check r})({\check r}\otimes I_{V})=(I_{V}\otimes {\check r})({\check r}\otimes 
I_{V})(I_{V}\otimes {\check r}).\]
Notice also that ${\check r}^{2}=I_{V \otimes V}$ the identity matrix, because $\check r$ is involutive.

For set-theoretic solutions it is thus convenient to use the matrix notation:
\begin{equation}
{\check r}=\sum_{x, y\in X} e_{x, \sigma_x(y)}\otimes e_{y, \tau_y(x)}. \label{brace1}
\end{equation}
Define also, $r={\cal P}{\check r}$, where ${\cal P} = \sum_{x, y \in X} e_{x,y} \otimes e_{y,x}$ 
is the permutation operator,  consequently
${ r}=\sum_{x, y\in X} e_{y,\sigma_x(y)}\otimes e_{x, \tau_y(x)}.$
The Yangian is a special case: $\check r(x,z|y,w)= \delta_{z,y}\delta_{w,x} $.
We note that in this study we focus on involutive, non-degenerate solutions of the braid equation.

Recall first the Yang-Baxter equation \cite{Baxter, Yang} in the braid form ($\delta = \lambda_1 - \lambda_2$):
\begin{equation}
\check R_{12}(\delta)\ \check R_{23}(\lambda_1)\ \check R_{12}(\lambda_2) = \check R_{23}(\lambda_2)\
 \check R_{12}(\lambda_1)\ \check R_{23}(\delta), \label{YBE1}
\end{equation}
where $\check R \in \mbox{End}(V\otimes V)$ and $\lambda_{1,2}$ are complex numbers.

We focus here on brace solutions\footnote{All, finite, non-degenerate,  involutive, 
set-theoretic solutions of the YBE (\ref{brace1}) are coming from braces (Theorem \ref{Rump}),
therefore we will call such solutions {\it  brace solutions}.}  of the YBE, given by (\ref{brace1}) 
and the Baxterized solutions (for a more detailed discussion we refer the interested reader 
to \cite{DoiSmo, DoiSmo2}): 
\begin{equation}
\check R(\lambda) = \lambda \check r +I_{V^{\otimes 2}}. \label{braid1}
\end{equation}
Let also $R = {\cal P} \check R$, (recall the permutation operator 
${\cal P} = \sum_{x, y\in X} e_{x,y}\otimes e_{y,x}$),  
then the following basic properties for $R$ matrices coming 
from braces were shown in \cite{DoiSmo}:\\

\noindent {\bf Basic Properties.} {\em  
The brace $R$-matrix satisfies the following fundamental properties:}
\begin{eqnarray}
&&  R_{12}(\lambda)\  R_{21}(-\lambda) = (-\lambda^2 +1) I_{V^{\otimes 2}}, 
~~~~~~~~~~~~~
\mbox{{\it Unitarity}} \label{u1}\\
&&  R_{12}^{t_1}(\lambda)\ R_{12}^{t_2}(-\lambda -{\cal N}) = 
\lambda(-\lambda -{\cal N})I_{V^{\otimes 2}}, 
~~~~~\mbox{{\it Crossing-unitarity}} \label{u2}\\
&& R_{12}^{t_1 t_2}(\lambda) = R_{21}(\lambda), \label{tt}\nonumber
\end{eqnarray}
{\it where $^{t_{1,2}}$ denotes transposition on the first, second space respectively, 
and recall ${\cal N}$ is the same as the cardinality of the set  $X$.}

\subsection{The Quantum Algebra associated to braces}

\noindent 
Given a solution of the Yang-Baxter equation, the quantum algebra is defined via the 
fundamental relation \cite{FadTakRes} (known as the RTT relation):
\begin{equation}
\check R_{12}(\lambda_1 -\lambda_2)\ L_1(\lambda_1)\ L_2(\lambda_2) = L_1(\lambda_2)\ L_2(\lambda_1)\ 
\check R_{12}(\lambda_1 -\lambda_2). \label{RTT}
\end{equation}
$\check R(\lambda) \in \mbox{End}(V\otimes V)$, $\ L(\lambda) \in 
\mbox{End}(V) \otimes {\mathfrak A}$, where ${\mathfrak A}$\footnote{Notice that in $L$ 
in addition to the indices 1 and 2 in (\ref{RTT}) there is also an implicit ``quantum index'' $n$ associated to ${\mathfrak A},$ 
which for now is omitted, i.e. one writes $L_{1n},\ L_{2n}$.} is the quantum algebra defined by (\ref{RTT}). 
We shall focus henceforth on solutions associated to braces only given by (\ref{braid1}), (\ref{brace1}). 
The defining relations of the coresponding 
quantum algebra were derived in \cite{DoiSmo}:\\

\noindent {\it The quantum algebra associated to the brace $R$ matrix  (\ref{braid1}), (\ref{brace1}) 
is defined by generators $L^{(m)}_{z,w},\ z, w \in X$, and defining relations }
\begin{eqnarray}
L_{z,w}^{(n)} L_{\hat z, \hat w}^{(m)} - L_{z,w}^{(m)} L_{\hat z, \hat w}^{(n)} &=& 
L^{(m)}_{z, \sigma_w(\hat w)} L^{(n+1)}_{\hat z,\tau_{\hat w}( w)}- L^{(m+1)}_{z, \sigma_w(\hat w)} 
L^{(n)}_{\hat z, \tau_{\hat w}( w)}\nonumber\\ &-& L^{(n+1)}_{ \sigma_z(\hat z),w} 
L^{(m)}_{\tau_{\hat z}( z), \hat w }+ L^{(n)}_{ \sigma_z(\hat z, )w}  L^{(m+1)}_{\tau_{\hat z}( z), \hat w}. \label{fund2}
\end{eqnarray}

The proof is based on the fundamental relation (\ref{RTT}) and the form of the brace $R$- matrix 
(for the detailed proof see \cite{DoiSmo}). Recall also that
in the index notation we define $\check R_{12} = \check R \otimes \mbox{id}_{\mathfrak A}$:
\begin{eqnarray}
&&  L_1(\lambda) = \sum_{z, w \in X} e_{z,w} \otimes I_V \otimes L_{z,w}(\lambda),\ \quad  L_2(\lambda)= \sum_{z, w \in X}
I_V  \otimes  e_{z,w}  \otimes L_{z,w}(\lambda).  \label{def}
\end{eqnarray} 
The exchange relations among the various generators of the affine algebra 
are derived below via (\ref{RTT}). Let us express $L$ as a formal power series expansion 
$L(\lambda) = \sum_{n=0}^{\infty} {L^{(n)} \over \lambda^n}$.
Substituting  expressions (\ref{braid1}), and the $\lambda^{-1}$ expansion in (\ref{RTT}) we obtain
the defining relations of the quantum algebra associated 
to a brace $R$-matrix (we focus on terms $\lambda_1^{-n} \lambda_2^{-m}$):
\begin{eqnarray}
&&  \check r_{12} L_{1}^{(n+1)} L_2^{(m)} -\check  r_{12} L_1^{(n)} L_2^{(m+1)} +  L_1^{(n)} L_2^{(m)} \nonumber\\
&&  = L_1^{(m)} L_{2}^{(n+1)} \check r_{12} -  L_1^{(m+1)} L_2^{(n)}\check r_{12} +  L_1^{(m)} L_2^{(n)}. \label{fund}
\end{eqnarray}
The latter relations immediately lead to the quantum algebra relations (\ref{fund2}), after recalling:
$
L_{1}^{(k)}=\sum_{x,y\in X}e_{x,y}\otimes I_V \otimes  L_{x,y}^{(k)},$ $ L_{2}^{(k)}=\sum_{x,y\in X} I_V \otimes 
e_{x,y}\otimes  L^{(k)}_{x,y}, \nonumber
$
and $\check r_{12} = \check r \otimes  \mbox{id}_{\mathfrak A },$ 
$L^{(k)}_{x,y} $ are the generators of the associated quantum algebra.
The quantum algebra is also equipped with a co-product $\Delta: {\mathfrak A} \to {\mathfrak A} \otimes {\mathfrak A}$ \cite{FadTakRes, Drinfeld}. Indeed, we define 
\begin{equation}
{\mathrm T}_{1;23}(\lambda):= (I_V \otimes \Delta)L= L_{13}(\lambda) L_{12}(\lambda),\  
\end{equation}
which satisfies (\ref{RTT}) and is expressed as ${\mathrm T}_{1;23}(\lambda) = \sum_{x,y \in X} e_{x,y} \otimes \Delta(L_{x,y}(\lambda)).$

\begin{rem} In the special case $\check r ={\cal P}$ the ${\cal Y}(\mathfrak {gl}_{\cal N})$ algebra is recovered (see for instance \cite{yangians} for a more detailed account on Yangians):
\begin{equation}
\Big [ L_{i,j}^{(n+1)},\ L_{k,l}^{(m)}\Big ] -\Big [ L_{i,j}^{(n)},\ L_{k,l}^{(m+1)}\Big ] = L_{k,j}^{(m)}L_{i,l}^{(n)}- L_{k,j}^{(n)}L_{i,l}^{(m)}. \label{fund2b}
\end{equation}
\end{rem}

The next natural step is  the classification of solutions of the fundamental relation (\ref{RTT}), 
for the brace quantum algebra.  A first step towards this goal  will be to examine  the simplest non-trivial fundamental object $L(\lambda)= L_0 + {1\over\lambda} L_1,$ and search for finite and infinite representations of the respective elements.  In the case of the ${\cal Y}(\mathfrak {gl}_{\cal N})$ an analogous object is $L(\lambda)= I_{V^{\otimes 2}}+ {1\over\lambda} {\mathbb P}$ where the elements of the matrix 
${\mathbb P}_{i,j}$ satisfy the $\mathfrak {gl}_{\cal N}$ algebraic relations. 
The classification of $L$-operators will allow the identification of new classes of quantum  integrable systems, such 
as  the analogues of Toda chains or deformed boson models.

We  briefly discuss below the existence of an admissible Drinfeld twist for invlolutive, non-degenerate, set-theoretic solutions of the YBE. 
Indeed, one of the most significant results in the case of involutive solutions of the YBE is their connection with the Yangian solution via a suitable admissible twist.  From Proposition 3.3 in \cite{DoiSmo2} we can extract explicit forms for the twist 
$F \in\mbox{End}({\mathbb C}^{\cal N}) \otimes\mbox{End}({\mathbb C}^{\cal N}) $ and state 
the following Proposition, which is Proposition 3.10 in \cite{Doikoutw}.
\begin{pro}{\label{twistlocal} (\cite{DoiSmo2,  Doikoutw})} 
Let $\check r = \sum_{x,y \in X} e_{x, \sigma_x(y)} \otimes e_{y, \tau_y(x) }$ be
 the set-theoretic solution of the braid YBE,  ${\cal P}$ is the permutation 
operator and $\hat V_k,\ V_k$ are their respective eigenvectors.  
Let $F^{-1} = \sum_{k=1}^{{\cal N}^2} \hat V_k\  V_k^T$ 
be the transformation (twist), such that $\check r = F^{-1} {\cal P} F.$
Then the twist can be explicitly expressed as 
$F = \sum_{x\in X} e_{x,x} \otimes {\mathbb V}_x$, where we define ${\mathbb V}_x =\sum_{y \in X} e_{\sigma_{x}(y), y}$. 
\end{pro}
For a detailed proof of the Proposition we refer the interested reader to \cite{DoiSmo2} and \cite{Doikoutw}.
However, by recalling that  $r = {\cal P} \check r,$ 
and using the fact that $\sigma_x,\ \tau_y$ 
are bijections, we confirm by direct computation that $(F^{(op)})^{-1} F= 
\sum_{x, \in X} e_{y, \sigma_x(y)} \otimes e_{x, \tau_y(x)}  =r,$ where we define $F^{op}:= {\cal P} F {\cal P}$ or in the index notation $F_{12}^{op} = F_{21}.$ 
 
\begin{rem} Let the Baxterized  solution of the YBE be $R(\lambda) = \lambda   r + {\cal P}.$ 
If $r$ satisfies the YBE and $r_{12} r_{21} =I$
then the Baxterized $R(\lambda)$ matrix also satisfies the YBE.
If $r= {\cal P} \check r$ is the set-theoretic solution of the YBE then, $R_{12}(\lambda)=F^{-1}_{21} (R_Y)_{12}(\lambda) F_{12}$, 
where $R_Y(\lambda) = \lambda I_V + {\cal P}$ is the Yangian $R$-matrix. 
This immediately follows from the form 
$R_Y(\lambda) = \lambda I+ {\cal P}$, 
and the property $F_{21}^{-1} {\cal P}_{12} {\cal F}_{12} = {\cal P}_{12}$. Note also  
that the twist is not uniquely defined, for instance an alternative twist is of the form 
$G= \sum_{x,y\in X}e_{\tau_{y}(x), x} \otimes e_{y,y} $, 
and $\sum_{x, \in X} e_{y, \sigma_x(y)} \otimes e_{x, \tau_y(x)}=  G_{21}^{-1} G_{12},$ see \cite{Doikoutw}.
\end{rem}

Although we will not extend our discussion further on Drinfeld's twist, it is worth noting that the admissibility of the twist was shown in \cite{Doikoutw}, whereas in \cite{Doikoutw, DoGhVl} it was proven that the quantum algebra coming from set-theoretic Baxterized solutions is in fact a quasi-bialgebra, and the twist turns the quasi-bialgebra to the Yangian Hopf algebra, as expected from Proposition \ref{twistlocal}.  For a detailed discussion on these fundamental issues we refer the interested reader to \cite{Doikoutw,  DoGhVl}.
We should also note that the discovery of the twist provides important information regarding the derivation of the spectrum of the associated quantum integrable systems, especially the ones with  special open boundary conditions. This issue will be addressed in detail separately in a future work.

\section{Set-theoretic reflection equation}

\noindent We shall focus in this section on the set-theoretic analogue of the reflection equation.
Let $ (X, \check r)$ be a non-degenerate set-theoretic solution to the Yang-Baxter equation. 
A map $k : X\to X$ is a reflection of $(X, \check r)$ if it satisfies 
\begin{equation}
		\check r(k\times id_X) \check r (k\times id_X)=(k\times id_X) \check r(k\times id_X) \check r. \label{re1}
	\end{equation}
We say that $k$ is a set-theoretic solution to the reflection equation.  We also say that  $k$ is involutive if $k(k(x))  = x$.

Examples of functions $k$ satisfying the reflection equation related to  braces can be found in \cite{SmoVenWes, Katsa, DeCommer}.
Recall that this set-theoretical version of the reflection equation  together with the first examples of solutions first appeared in the work 
of Caudrelier and Zhang \cite{CauZha}.
Solutions of the set-theoretic reflection equation using braces have been studied in \cite{SmoVenWes, Katsa}.
The main Theorem 1.8 of \cite{SmoVenWes} is stated below.

\begin{rem} We note that in \cite{SmoVenWes} the ``dual'' reflection equation is considered,  i.e. 
\begin{equation}
		\check r(id_X\times k) \check r (id_X\times k)=(id_X\times k) \check r(id_X\times k)\check r, \label{re1d}
	\end{equation}
thus in our findings  below $\sigma,$ $\tau $ are interchanged compared to the results of \cite{SmoVenWes}.
\end{rem}

\begin{thm}\label{Prop7} Let $(X,\check r)$ $\check r : X \times X \to X \times X$ be an involutive, 
non-degenerate solution of the braid equation,
$\check r(x,y) = (\sigma  _{x}(y), \tau _{x}(y)).$
Let also the map $k: X\rightarrow X,$ then $k$ satisfies the reflection equation (\ref{re1}) if and only if 
\begin{equation}
\tau_{\tau _{y}(x)}(k(\sigma _{x}(y))) =\tau_{\tau _{y}(k(x))}(k(\sigma _{k(x)}(y))). \label{cc1}
\end{equation}
\end{thm}
\begin{proof} The proof is presented in Theorem 1.8  in \cite{SmoVenWes}, 
when we interchange $\sigma $ with $\tau $. 
\end{proof}
\begin{rem}
Let $(X,\check r)$ be an involutive, non-degenerate solution of the braid equation where we denote 
${\check r}(x,y)=(\sigma _{x}(y), \tau _{y}(x))$, and let $k:X\rightarrow X$ be a function. 
We say that $k$ is $\tau $-equivariant if for every $x, y\in X$ we have 
\[\tau _{x}(k(y))=k(\tau _{x}(y)).\]
Every function $k:X\rightarrow X$  satisfying $k(\tau _{y}(x))=\tau_{y}(k(x))$
satisfies the set-theoretic reflection equation (see Theorem 1.9 in  \cite{SmoVenWes})). 
The proof for the latter statement is straightforward, indeed  if $k(\tau _{y}(x))=\tau_{y}(k(x)),$ then the LHS of (\ref{cc1}) becomes
$k(\tau_{\tau _{y}(x)}(\sigma _{x}(y)))= k(y), $ where we have used the fact that $\check r$ is involutive i.e.  
$\tau_{\tau _{y}(x)}(\sigma _{x}(y)) =y,$ similarly  the RHS  of (\ref{cc1}) is $k(y),$ which shows that $k$ is a reflection. 
\end{rem}

\begin{example}
 In \cite{Katsa} central elements were used to to define 
$\tau $-equivariant functions in an analogous way- as $k(x)=\tau_{c}(x)$, where $c$ is central, i.e.  for every $x\in X,$ $c\circ x =x \circ c$. 
\end{example}

\subsection{Reflection algebra from set-theoretic solutions}

\noindent 
We use the matrix notation introduced in Subsection 1.1,  then the reflection matrix $K$ is  an ${\cal N} \times {\cal N}$ matrix
represented as: ${\mathrm k} = \sum_{x\in X} e_{x, k(x)},$
and satisfies the constant reflection equation:
\begin{equation}
\check r(  {\mathrm k}  \otimes I_V)  \check r ( {\mathrm k}\otimes  I_V )=({\mathrm k}\otimes  I_V ) 
\check r({\mathrm k}\otimes  I_V ) \check r.
\end{equation}

We introduce quadratic algebras associated to the classification
of boundary conditions in quantum integrable models., i.e. we consider generic $\lambda$ dependent $R$-matrices. 
To define these quadratic algebras
in addition to the $R$-matrix we also need to introduce the $K$-matrix, 
which physically describes the interaction of particle-like excitations displayed by 
the quantum integrable  system, with the boundary of the system. 
The $K$-matrix satisfies \cite{Cherednik, Sklyanin}:    
\begin{equation}
R_{12}(\lambda_1 - \lambda_2) {\mathbb K}_1(\lambda_1) \hat R_{12}(\lambda_1+\lambda_2) {\mathbb K}_{2}(\lambda_2) =  
{\mathbb K}_{2}(\lambda_2)  \hat R_{21}(\lambda_1+\lambda_2) {\mathbb K}_1(\lambda_1 )R_{21}(\lambda_1 - \lambda_2), \label{RE2}
\end{equation}
where we define in general $A_{21} = {\cal P}_{12} A_{12} {\cal P}_{12}$.
We focus on  the case where $\hat R_{12} (\lambda)= R_{12}^{-1}(-\lambda)\propto R_{21}(\lambda)$,   i.e. 
we consider the boundary Yang-Baxter or reflection equation \cite{Cherednik, Sklyanin},  and  we recall that $\check R = {\cal P} R$ then 
the reflection equation is expressed in the braid form as
\begin{equation}
\check R_{12}(\lambda_1-\lambda_2){\mathbb K}_1(\lambda_1) \check R_{12}(\lambda_1 +\lambda_2) {\mathbb K}_1(\lambda_2)= 
{\mathbb K}_1 (\lambda_2) \check R_{12}(\lambda_1 +\lambda_2) {\mathbb K}_1(\lambda_1) \check R_{12}(\lambda_1-\lambda_2). \label{RE}
\end{equation}
As in the case of the Yang-Baxter equation, where representations of the $A$-type Hecke algebra are associated to solutions of the
Yang-Baxter equation \cite{Jimbo}, via the Baxterization process,  representations of the $B$-type Hecke algebra provide solutions 
of the reflection equation \cite{LevyMartin, DoikouMartin}.

\noindent We shall discuss in more detail now the algebra associated to the quadratic equation (\ref{RE2}). 
A solution of the quadratic equation (\ref{RE2}) is of the form \cite{Sklyanin}
\begin{equation}
{\mathbb K}(\lambda |\theta_1)= L(\lambda -\theta_1)\  \big ( K(\lambda) \otimes \mbox{id}\big )\  \hat L(\lambda +\theta_1), \label{rep1}
\end{equation}
where $L(\lambda) \in\mbox{End}(V) \otimes {\mathfrak A}$ satisfies the RTT  relation (\ref{RTT})
and $K(\lambda) \in \mbox{End}(V) $ is a $c$-number solution of the quadratic equation (\ref{RE2})  for some $R(\lambda)\in\mbox{End}(V \otimes V),$ solution of the Yang-Baxter equation.  
We also define (in the index notation, see also Footnote 2)
\begin{eqnarray}
 \hat L_{1n}(\lambda) = L_{1n}^{-1}(-\lambda) \nonumber
\end{eqnarray}
The quadratic algebra  ${\mathfrak B}$ defined by (\ref{RE2}) is a left co-ideal of the quantum algebra ${\mathfrak A}$ for
a given $R$-matrix (see also e.g. \cite{Sklyanin, DeMa,  MoRa, Doikou2}), i.e. the algebra  is endowed with a co-product 
$\Delta: {\mathfrak B} \to {\mathfrak B} \otimes {\mathfrak  A}$ \cite{Sklyanin}. Indeed, we define (in the index notation)
\begin{equation}
{\mathbb T}_{0;12}(\lambda|\theta_1,\theta_2) = L_{02}(\lambda-\theta_2) {\mathbb K}_{01}(\lambda|\theta_1)\hat L_{02}(\lambda+\theta_2), \label{rep3}
\end{equation}
where ${\mathbb K}(\lambda|\theta_1)$ is given in (\ref{rep1}) and in the index notation ${\mathbb K}_{01}(\lambda|\theta_1) = L_{01}(\lambda-\theta_1) K_0(\lambda) \hat L_{01}(\lambda+\theta_1)$.
Let also ${\mathbb K}_{01}(\lambda|\theta_1) = \sum_{a,b =1}^{\cal N}e_{a,b}\otimes {\mathbb K}_{a,b}(\lambda|\theta_1)\otimes\mbox{id}_{\mathfrak A} $,  $L_{02} = \sum_{a,b=1}^{\cal N}e_{a,b} \otimes \mbox{id}_{\mathfrak A} \otimes L_{a,b}(\lambda)$ and
 ${\mathbb T}_{0;12}(\lambda|\theta_1,\theta_2)= \sum_{a,b=1}^{\cal N}e_{a,b} \otimes \Delta({\mathbb K}_{a, b}(\lambda|\theta_1,\theta_2)),$ then via expression (\ref{rep3}):
\begin{equation}
\Delta({\mathbb K}_{a,b}(\lambda|\theta_1, \theta_2) )= \sum_{k,l} {\mathbb K}_{k,l}(\lambda|\theta_1) \otimes L_{a,k}(\lambda -\theta_2) \hat L_{l,b}(\lambda+\theta_2), \label{rep2}
\end{equation}
where the elements ${\mathbb K}_{k,l}(\lambda|\theta_1)$  can be also re-expressed in terms of the elements of the $c$-number matrix $K$ and $L$  when considering the realization (\ref{rep1}).

In our analysis  for the rest of the present subsection we shall  be considering  $\check R(\lambda) = \lambda \check r + I_V^{\otimes 2}$, 
where $\check r$ provides a representation of the symmetric group.

\begin{pro} \label{defin} Let $\check R(\lambda)= \lambda\check r + I_V^{\otimes 2}$, where $\check r$ 
provides a tensor realization of 
the Hecke algebra ${\cal H}_N(q=1)$, and let ${\mathbb K}(\lambda)$ satisfy the quadratic equation (\ref{RE2}).
Let also ${\mathbb K}(\lambda) = 
\sum_{n=0}^{\infty}{{\mathbb K}^{(n)} \over \lambda^n}$ and $~{\mathbb K}^{(n)} = \sum_{z, w\in X} e_{z,w}\otimes {\mathbb K}_{z,w}^{(n)},$ where ${\mathbb K}_{z,w}^{(n)}$ are the generators of the quadratic algebra defined by (\ref{RE2}). 
The exchange relations among the
quadratic algebra generators are encoded in:
\begin{eqnarray}
& & \check r_{12}{\mathbb K}_1^{(n+2)} \check r_{12} {\mathbb K}_1^{(m)} - \check r_{12} {\mathbb K}_1^{(n)} 
\check r_{12} {\mathbb K}_1^{(m+2)} + 
\check r_{12} {\mathbb K}_1^{(n+1)} {\mathbb K}_1^{(m)}\nonumber\\ 
& & - \check r_{12} 
{\mathbb K}_1^{(n)}  {\mathbb K}_1^{(m+1)}
+  {\mathbb K}_1^{(n+1) }\check r_{12} {\mathbb K}_1^{(m)} +  {\mathbb K}_1^{(n) }\check r_{12}
{\mathbb K}_1^{(m+1)}  + {\mathbb K}_1^{(n)} {\mathbb K}_1^{(m)} \nonumber\\
&= & {\mathbb K}_1^{(m)} \check r_{12} {\mathbb K}_1^{(n+2)}\check r_{12} - {\mathbb K}_1^{(m+2)} 
\check r_{12} {\mathbb K}_1^{(n)}\check r_{12} + 
{\mathbb K}_1^{(m)}  {\mathbb K}_1^{(n+1)} \check r_{12} \nonumber\\ 
&& - {\mathbb K}_1^{(m+1)} {\mathbb K}_1^{(n)}\check r_{12}+  {\mathbb K}_1^{(m+1)} 
\check r_{12} {\mathbb K}_1^{(n)}+  
{\mathbb K}_1^{(m)} \check r_{12} {\mathbb K}_1^{(n+1)} + {\mathbb K}_1^{(m)}{\mathbb K}_1^{(n)}. \nonumber \\
&& \label{RABasic}
\end{eqnarray}
\end{pro}
\begin{proof}
First we act from the left and right of (\ref{RE2}) with the permutation operator ${\cal P},$ then (\ref{RE2}) becomes
\begin{equation}
\check R_{12}(\lambda_1-\lambda_2) {\mathbb K}_1(\lambda_1) \check R_{12}(\lambda_1+\lambda_2)
{\mathbb K}_1(\lambda_2)={\mathbb K}_1(\lambda_2)  \check R_{12}(\lambda_1+\lambda_2){\mathbb K}_1(\lambda_1)\check R_{12}(\lambda_1-\lambda_2), \label{RE3}
\end{equation}
where $\check R(\lambda_1 \pm \lambda_2) = (\lambda_1 \pm \lambda_2)\check r + I_V^{\otimes 2 }$ , and we recall that ${\mathbb K}(\lambda_i) = \sum_{n=0}^{\infty}{{\mathbb K}^{(n)} \over \lambda_i^n}$ ($i\in \{1,\ 2\}$).
We substitute the above expressions in  (\ref{RE3}),  and we gather terms proportional to $\lambda_1^{-n} \lambda_2^{-m}$, $n, m\geq 0$ in the LHS and RHS of (\ref{RE3}), which lead to  (\ref{RABasic}). Recalling also that in general $A_{12} = A \otimes \mbox{id}_{\mathfrak A}$, 
$~{\mathbb K}^{(n)}_1 = \sum_{z, w\in X} e_{z,w} \otimes I_V\otimes {\mathbb K}_{z,w}^{(n)},$ 
and substituting the latter  expressions in (\ref{RABasic}) we obtain the exchange relations among the 
generators ${\mathbb K}_{z,w}^{(n)}$, which are particularly involved and are omitted here.
\end{proof}

It is useful for the following Corollaries to focus on terms proportional to $\lambda_1^2 \lambda_2^{-m}$ and $\lambda_1 \lambda_2^{-m}$ 
(or equivalently  $\lambda_2^2 \lambda_1^{-m}$ and $\lambda_2 \lambda_1^{-m}$) in the $\lambda_{1,2}$ expansion of the quadratic algebra, and obtain
\begin{eqnarray}
&& \check r_{12} {\mathbb K}_1^{(0)}\check r_{12} {\mathbb K}_1^{(m)}  =  
{\mathbb K}_1^{(m)} \check r_{12}{\mathbb K}_1^{(0)}\check  r_{12}  \label{RABasic2} \\
&& \check r_{12} {\mathbb K}_1^{(1)}\check r_{12} {\mathbb K}_1^{(m)} + {\mathbb K}_1^{(0)} \check r_{12} {\mathbb K}_1^{(m)}
+  \check r_{12} {\mathbb K}_1^{(0)} {\mathbb K}_1^{(m)}
=  \label{RABasic2b} \\
&& {\mathbb K}_1^{(m)} \check r_{12}{\mathbb K}_1^{(1)}\check  r_{12}    +  {\mathbb K}_1^{(m)} \check r_{12}  {\mathbb K}_1^{(0)}
+  {\mathbb K}_1^{(m)}  {\mathbb K}_1^{(1)}\check  r_{12}.\nonumber
\end{eqnarray}

\begin{cor} \label{Proposition 6}{\it A finite non-abelian sub-algebra of the reflection algebra  exists, 
realized by the elements of ${\mathbb K}^{(1)}$ when ${\mathbb K}^{(0)} \propto I_V$.}
\end{cor}
\begin{proof}
We focus on terms proportional $\lambda_1^2 \lambda_2^{-m}$ and $\lambda_1 \lambda_2^{-m}$ (\ref{RABasic2}),  
(\ref{RABasic2b}) in the case of the reflection algebra:
\begin{eqnarray}
&& \Big [\check r_{12} {\mathbb K}_1^{(0)} \check r_{12},\ {\mathbb K}_1^{(m)} \Big ] =0 \label{rela1}  \\
&& \Big [ \check r_{12} {\mathbb K}_1^{(1)}\check r_{12},\ {\mathbb K}_1^{(m)} \Big ]= \label{rela2}\\
&&  \ {\mathbb K}_1^{(m)}  {\mathbb K}_1^{(0)} \check r_{12}  + {\mathbb K}_1^{(m)}  
\check r_{12}  {\mathbb K}_1^{(0)} - {\mathbb K}_1^{(0)} \check r_{12} {\mathbb K}_1^{(m)}
-  \check r_{12} {\mathbb K}_1^{(0)}{\mathbb K}_1^{(m)}.  \nonumber
\end{eqnarray}
Notice that due to (\ref{rep1})  in the case of the reflection algebra ${\mathbb K}^{(0)} 
\propto  I_V$ when the $c$-number matrix $K  \propto I_V$.
For $m=1$ equation (\ref{rela2}) provides the defining relations of a finite sub-algebra of the reflection algebra generated 
by ${\mathbb K}^{(1)}_{x,y}$ .
\end{proof}
With this we conclude our presentation on the algebraic content of both set theoretic Yang-Baxter and reflection equations.

\section{Conclusions}
\noindent We presented fundamental findings on involutive, non-degenerate solutions of the set-theoretic Yang-Baxter and reflection equations. 
We  recalled the notion of braces and showed a number of key properties necessary for the solution 
of the set-theoretic Yang-Baxter equation.  We then identified
the associated quantum algebra for parameter dependent  set-theoretic solutions and we briefly discussed the notion of the Drinfeld twist for involutive solutions and their relation to the Yangian.
In the second part we focused on reflections and we derived the associated reflection algebra for $R$-matrices being  
Baxterized solutions of the symmetric group and showed that there exists a ``reflection'' finite sub-algebra for some special choice of reflection maps

The next important step is the diagonalizaton of the constructed spin chain like systems \cite{DoiSmo, DoiSmo2} 
for open and periodic system.  This is a challenging problem,  however  the discovery of the associated Drinfeld twist 
\cite{Doikoutw,  DoGhVl} is a first important step towards this direction.  The deeper understanding of the associated Drinfeld twist and the properties of set-theoretic solutions, especially the involutive ones, will provide the necessary background for the derivation of the universal $R$-matrix in this context.

\subsection*{Acknowledgments}

\noindent  I am indebted to A. Smoktunowicz  and B. Rybolowicz for useful discussions and prior collaborations on the subject.
Support from the EPSRC research grant  EP/V008129/1 is also acknowledged.

\end{document}